\documentclass[12pt,fleqn]{article}

\usepackage{amsmath,amssymb,amsthm}
\usepackage{geometry}
\usepackage[pdfpagemode=UseNone,pdfstartview=FitH,hypertex]{hyperref}

\newcommand{\abs}[1]{\left|#1\right|}
\newcommand{\bdry}[1]{\partial #1}

\newcommand{\closure}[1]{\overline{#1}}
\newcommand{\dint}{\ds{\int}}
\newcommand{\ds}[1]{\displaystyle #1}
\newcommand{\eps}{\varepsilon}
\newcommand{\half}{\frac{1}{2}}
\newcommand{\norm}[2][]{\left\|#2\right\|_{#1}}
\renewcommand{\O}{\text{O}}
\renewcommand{\o}{\text{o}}
\newcommand{\PS}[1]{$(\text{PS})_{#1}$}
\newcommand{\pnorm}[2][]{\if #1'' \left|#2\right|_p \else \left|#2\right|_{#1} \fi}
\newcommand{\QED}{\mbox{\qedhere}}
\newcommand{\R}{\mathbb R}
\newcommand{\seq}[1]{\left(#1\right)}
\newcommand{\set}[1]{\left\{#1\right\}}
\newcommand{\vol}[1]{\left|#1\right|}

\DeclareMathOperator{\divg}{div}

\newenvironment{enumroman}{\begin{enumerate}

}{\end{enumerate}}

\newtheorem{lemma}{Lemma}[section]
\newtheorem{theorem}[lemma]{Theorem}

\numberwithin{equation}{section}

\title{\bf Positive solutions of semipositone elliptic problems with critical Trudinger-Moser nonlinearities}
\author{\bf Kanishka Perera\\
Department of Mathematical Sciences\\
Florida Institute of Technology\\
Melbourne, FL 32901, USA\\
\em kperera@fit.edu\\
[\bigskipamount]
\bf Inbo Sim\\
Department of Mathematics\\
University of Ulsan\\
Ulsan 44610, Republic of Korea\\
\em ibsim@ulsan.ac.kr}
\date{}

\begin{document}

\maketitle

\begin{abstract}
We prove the existence of a positive solution to a semipositone $N$-Laplacian problem with a critical Trudinger-Moser nonlinearity. The proof is based on obtaining uniform $C^{1,\alpha}$ a priori estimates via a compactness argument. Our result is new even in the semilinear case $N = 2$, and our arguments can easily be adapted to obtain positive solutions of more general semipositone problems with critical Trudinger-Moser nonlinearities.
\end{abstract}

\section{Introduction}

Elliptic problems with critical Trudinger-Moser nonlinearities have been widely investigated in the literature. We refer the reader to the survey paper of de Figueiredo et al.\! \cite{MR2772124} for an overview of recent results on Trudinger-Moser type inequalities and related critical problems. A model critical problem of this type is
\[
\left\{\begin{aligned}
- \Delta_N\, u & = \lambda\, |u|^{N-2} u\, e^{\beta\, |u|^{N'}} && \text{in } \Omega\\[10pt]
u & = 0 && \text{on } \bdry{\Omega},
\end{aligned}\right.
\]
where $\Omega$ is a smooth bounded domain in $\R^N,\, N \ge 2$, $\Delta_N\, u = \divg \left(|\nabla u|^{N-2}\, \nabla u\right)$ is the $N$-Laplacian of $u$, $N' = N/(N - 1)$, and $\lambda, \beta > 0$. This problem is a natural analog of the Br\'{e}zis-Nirenberg problem for the $p$-Laplacian in the borderline case $p = N$, where the critical growth is of exponential type and is governed by the Trudinger-Moser inequality
\begin{equation} \label{2}
\sup_{u \in W^{1,N}_0(\Omega),\, \norm{u} \le 1} \int_\Omega e^{\alpha_N |u|^{N'}} dx < \infty.
\end{equation}
Here $W^{1,N}_0(\Omega)$ is the usual Sobolev space with the norm
\[
\norm{u} = \left(\int_\Omega |\nabla u|^N\, dx\right)^{1/N},
\]
$\alpha_N = N \omega_{N-1}^{1/(N-1)}$, and $\omega_{N-1}$ is the area of the unit sphere in $\R^N$ (see Trudinger \cite{MR0216286} and Moser \cite{MR0301504}). A result of Adimurthi \cite{MR1079983} gives a positive solution of this problem for $\lambda \in (0,\lambda_1)$, where $\lambda_1 > 0$ is the first Dirichlet eigenvalue of $- \Delta_N$ in $\Omega$ (see also do {\'O} \cite{MR1392090}). Theorem 1.4 in de Figueiredo et al.\! \cite{MR1386960,MR1399846} gives a nontrivial solution for $\lambda \ge \lambda_1$ in the semilinear case $N = 2$. More recently, Yang and Perera \cite{MR3616328} obtained a nontrivial solution in the general quasilinear case $N \ge 3$ when $\lambda > \lambda_1$ is not an eigenvalue.

In the present paper we study the related semipositone problem
\begin{equation} \label{1}
\left\{\begin{aligned}
- \Delta_N\, u & = \lambda u^{N-1} e^{\beta u^{N'}} - \mu && \text{in } \Omega\\[5pt]
u & > 0 && \text{in } \Omega\\[5pt]
u & = 0 && \text{on } \bdry{\Omega},
\end{aligned}\right.
\end{equation}
where $\mu > 0$. Since $- \mu < 0$, $u = 0$ is not a subsolution of this problem, which makes finding a positive solution rather difficult (see Lions \cite{MR678562}). This compounds the usual difficulties arising from the lack of compactness associated with critical growth problems. Our main result here is that this problem has a weak positive solution for all sufficiently small $\mu$ when $\lambda < \lambda_1$.

\begin{theorem} \label{Theorem 1}
If $\lambda \in (0,\lambda_1)$, then there exists $\mu^\ast > 0$ such that for all $\mu \in (0,\mu^\ast)$, problem \eqref{1} has a weak solution $u_\mu \in C^{1,\alpha}_0(\closure{\Omega})$ for some $\alpha \in (0,1)$.
\end{theorem}

This result seems to be new even in the semilinear case $N = 2$. The outline of the proof is as follows. We consider the modified problem
\begin{equation} \label{3}
\left\{\begin{aligned}
- \Delta_N\, u & = \lambda f(u^+) - \mu\, g(u) && \text{in } \Omega\\[10pt]
u & = 0 && \text{on } \bdry{\Omega},
\end{aligned}\right.
\end{equation}
where $f(t) = t^{N-1} e^{\beta t^{N'}}$ for $t \ge 0$, $u^+(x) = \max \set{u(x),0}$, and
\[
g(t) = \begin{cases}
0, & t \le -1\\[5pt]
1 + t, & -1 < t < 0\\[5pt]
1, & t \ge 0.
\end{cases}
\]
Weak solutions of this problem coincide with critical points of the $C^1$-functional
\[
E_\mu(u) = \int_\Omega \left[\frac{|\nabla u|^N}{N} - \lambda F(u^+) + \mu\, G(u)\right] dx, \quad u \in W^{1,N}_0(\Omega),
\]
where
\[
F(t) = \int_0^t f(s)\, ds, \quad t \ge 0, \qquad G(t) = \int_0^t g(s)\, ds, \quad t \in \R.
\]
The functional $E_\mu$ satisfies the \PS{c} condition for all $c \ne 0$ satisfying
\[
c < \frac{1}{N} \left(\frac{\alpha_N}{\beta}\right)^{N-1} - \frac{\mu}{2} \vol{\Omega},
\]
where $\vol{\cdot}$ denotes the Lebesgue measure in $\R^N$, and it follows from the mountain pass theorem that $E_\mu$ has a uniformly positive critical level below this threshold for compactness for all sufficiently small $\mu > 0$ (see Lemmas \ref{Lemma 2} and \ref{Lemma 6}). This part of the proof is more or less standard. The novelty of the paper lies in the fact that the solution $u_\mu$ of the modified problem \eqref{3} thus obtained is positive, and hence solves our original problem \eqref{1}, if $\mu$ is further restricted. Note that this does not follow from standard arguments based on the maximum principle since the perturbation term $- \mu < 0$. This is precisely the main difficulty in finding positive solutions of semipositone problems as was pointed out in Lions \cite{MR678562}.

We will prove that for every sequence $\mu_j > 0,\, \mu_j \to 0$, a subsequence of $u_j = u_{\mu_j}$ is positive in $\Omega$. The idea of the proof is to show that a subsequence of $u_j$ converges in $C^1_0(\closure{\Omega})$ to a solution of the limit problem
\[
\left\{\begin{aligned}
- \Delta_N\, u & = \lambda f(u^+) && \text{in } \Omega\\[10pt]
u & = 0 && \text{on } \bdry{\Omega}.
\end{aligned}\right.
\]
This requires a uniform $C^{1,\alpha}_0(\closure{\Omega})$ estimate of $u_j$ for some $\alpha \in (0,1)$. It is well-known that each $u_j$ belongs to $C^{1,\alpha}_0(\closure{\Omega})$. However, proving that the sequence $\seq{u_j}$ remains bounded in $C^{1,\alpha}_0(\closure{\Omega})$ is a nontrivial task in the critical case. We will obtain the required estimate by proving the following compactness result, which is of independent interest.

\begin{theorem} \label{Theorem 2}
If $\mu_j > 0,\, \mu_j \to \mu \ge 0$, $\seq{u_j} \subset W^{1,N}_0(\Omega)$, and
\[
E_{\mu_j}(u_j) \to c, \qquad E_{\mu_j}'(u_j) \to 0
\]
for some $c \ne 0$ satisfying
\begin{equation} \label{4}
c < \frac{1}{N} \left(\frac{\alpha_N}{\beta}\right)^{N-1} - \frac{\mu}{2} \vol{\Omega},
\end{equation}
then a subsequence of $\seq{u_j}$ converges to a critical point of $E_\mu$ at the level $c$.
\end{theorem}

This theorem implies that
\[
\sup_j \int_\Omega e^{b\, |u_j|^{N'}} dx < \infty
\]
for all $b$ (see Lemma \ref{Lemma 8}). This together with the H\"{o}lder inequality implies that $f(u_j^+)$ is bounded in $L^s(\Omega)$ for all $s > 1$, so $\seq{u_j}$ is bounded in $L^\infty(\Omega)$ by Guedda and V{\'e}ron \cite[Proposition 1.3]{MR1009077}. The global regularity result in Lieberman \cite{MR969499} then gives the desired $C^{1,\alpha}_0(\closure{\Omega})$ estimate.

Theorem \ref{Theorem 2} is proved in Section \ref{Section 2} and Theorem \ref{Theorem 1} in Section \ref{Section 3}. In closing the introduction we remark that we have confined ourselves to the model problem \eqref{1} only for the sake of simplicity. The arguments given in this paper can easily be adapted to obtain positive solutions of more general semipositone problems with critical Trudinger-Moser nonlinearities.

\section{Proof of Theorem \ref{Theorem 2}} \label{Section 2}

In this section we prove Theorem \ref{Theorem 2}. First we collect some elementary estimates for easy reference.

\begin{lemma} \label{Lemma 1}
For all $t \ge 0$,
\begin{enumroman}
\item \label{Lemma 1 (i)} $F(t) \le \ds{\frac{N - 1}{\beta N}\, \frac{tf(t)}{t^{N/(N-1)}}}$,
\item \label{Lemma 1 (ii)} $F(t) \le F(1) + \dfrac{N - 1}{N(N + \beta - 1)}\, tf(t)$,
\item \label{Lemma 1 (iii)} $F(t) \le \dfrac{1}{N}\, tf(t)$,
\item \label{Lemma 1 (iv)} $F(t) \le \dfrac{1}{N}\, t^N + \dfrac{\beta}{N}\, t^{N^2/(N-1)} e^{\beta t^{N'}}$,
\item \label{Lemma 1 (v)} $F(t) \ge \ds{\frac{1}{N}\, t^N + \frac{\beta (N - 1)}{N^2}\, t^{N^2/(N-1)}}$.
\end{enumroman}
\end{lemma}

\begin{proof}
\ref{Lemma 1 (i)}. Integrating by parts,
\begin{align*}
F(t) & \le \frac{N - 1}{\beta N}\, t^{N-N/(N-1)} e^{\beta t^{N'}} - \frac{N - 2}{\beta} \int_0^t s^{N-N/(N-1)-1} e^{\beta s^{N'}}\, ds\\[3pt]
& \le \frac{N - 1}{\beta N}\, \frac{t^N e^{\beta t^{N'}}}{t^{N/(N-1)}}\\[3pt]
& = \frac{N - 1}{\beta N}\, \frac{tf(t)}{t^{N/(N-1)}}.
\end{align*}

\ref{Lemma 1 (ii)}. For $t \le 1$, $F(t) \le F(1)$. For $t > 1$, $F(t) = F(1) + \dint_1^t f(s)\, ds$. Integrating by parts,
\begin{align*}
\int_1^t f(s)\, ds & \le \frac{1}{N}\, t^N e^{\beta t^{N'}} - \frac{\beta}{N - 1} \int_1^t s^{N-1+N/(N-1)} e^{\beta s^{N'}}\, ds\\[3pt]
& \le \frac{1}{N}\, tf(t) - \frac{\beta}{N - 1} \int_1^t f(s)\, ds,
\end{align*}
and hence $\ds{\int_1^t f(s)\, ds \le \frac{N - 1}{N(N + \beta - 1)}\, tf(t)}$.

\ref{Lemma 1 (iii)}. Integrating by parts,
\begin{align*}
F(t) & = \frac{1}{N}\, t^N e^{\beta t^{N'}} - \frac{\beta}{N - 1} \int_0^t s^{N+N/(N-1)-1} e^{\beta s^{N'}}\, ds\\[3pt]
& \le \frac{1}{N}\, tf(t).
\end{align*}

\ref{Lemma 1 (iv)}. Since $e^t \le 1 + te^t$ for all $t \ge 0$,
\begin{align*}
F(t) & \le \int_0^t s^{N-1} \left(1 + \beta s^{N/(N-1)} e^{\beta s^{N'}}\right) ds\\[3pt]
& \le \frac{1}{N}\, t^N \left(1 + \beta t^{N/(N-1)} e^{\beta t^{N'}}\right).
\end{align*}

\ref{Lemma 1 (v)}. Since $e^t \ge 1 + t$ for all $t \ge 0$,
\begin{align*}
F(t) & \ge \int_0^t s^{N-1} \left(1 + \beta s^{N/(N-1)}\right) ds\\[3pt]
& = \frac{1}{N}\, t^N + \frac{\beta (N - 1)}{N^2}\, t^{N^2/(N-1)}. \QED
\end{align*}
\end{proof}

Next we prove the following lemma.

\begin{lemma} \label{Lemma 4}
If $\seq{u_j}$ is a sequence in $W^{1,N}_0(\Omega)$ converging a.e.\! to $u \in W^{1,N}_0(\Omega)$ and
\begin{equation} \label{5}
\sup_j \int_\Omega u_j^+ f(u_j^+)\, dx < \infty,
\end{equation}
then
\[
\int_\Omega F(u_j^+)\, dx \to \int_\Omega F(u^+)\, dx.
\]
\end{lemma}

\begin{proof}
For $M > 0$, write
\[
\int_\Omega F(u_j^+)\, dx = \int_{\{u_j^+ < M\}} F(u_j^+)\, dx + \int_{\{u_j^+ \ge M\}} F(u_j^+)\, dx.
\]
By Lemma \ref{Lemma 1} \ref{Lemma 1 (i)} and \eqref{5},
\[
\int_{\{u_j^+ \ge M\}} F(u_j^+)\, dx \le \frac{N - 1}{\beta NM^{N/(N-1)}} \int_\Omega u_j^+ f(u_j^+)\, dx = \O\! \left(\frac{1}{M^{N/(N-1)}}\right) \text{ as } M \to \infty.
\]
Hence
\[
\int_\Omega F(u_j^+)\, dx = \int_{\{u_j^+ < M\}} F(u_j^+)\, dx + \O\! \left(\frac{1}{M^{N/(N-1)}}\right),
\]
and the conclusion follows by first letting $j \to \infty$ and then letting $M \to \infty$.
\end{proof}

We will also need the following result of Lions \cite{MR834360} (see {\em Remark} I.18 $(i)$).

\begin{lemma} \label{Lemma 5}
If $\seq{u_j}$ is a sequence in $W^{1,N}_0(\Omega)$ with $\norm{u_j} = 1$ for all $j$ and converging a.e.\! to a nonzero function $u \in W^{1,N}_0(\Omega)$, then
\[
\sup_j \int_\Omega e^{b\, |u_j|^{N'}}\, dx < \infty
\]
for all $b < \alpha_N/(1 - \norm{u}^N)^{1/(N-1)}$.
\end{lemma}

We are now ready to prove Theorem \ref{Theorem 2}.

\begin{proof}[Proof of Theorem \ref{Theorem 2}]
We have
\begin{equation} \label{6}
E_{\mu_j}(u_j) = \frac{1}{N} \norm{u_j}^N - \lambda \int_\Omega F(u_j^+)\, dx + \mu_j \int_\Omega G(u_j)\, dx = c + \o(1)
\end{equation}
and
\begin{equation} \label{7}
E_{\mu_j}'(u_j)\, u_j = \norm{u_j}^N - \lambda \int_\Omega u_j^+ f(u_j^+)\, dx + \mu_j \int_\Omega u_j\, g(u_j)\, dx = \o(\norm{u_j}).
\end{equation}
Since
\[
\int_\Omega F(u_j^+)\, dx \le F(1) \vol{\Omega} + \frac{N - 1}{N(N + \beta - 1)} \int_\Omega u_j^+ f(u_j^+)\, dx
\]
by Lemma \ref{Lemma 1} \ref{Lemma 1 (ii)}, $\seq{\mu_j}$ is bounded, and
\begin{equation} \label{8}
\abs{\int_\Omega u_j\, g(u_j)\, dx} \le \int_\Omega |u_j|\, dx, \qquad \abs{\int_\Omega G(u_j)\, dx} \le \int_\Omega |u_j|\, dx,
\end{equation}
it follows from \eqref{6} and \eqref{7} that $\seq{u_j}$ is bounded in $W^{1,N}_0(\Omega)$. Hence a renamed subsequence converges to some $u$ weakly in $W^{1,N}_0(\Omega)$, strongly in $L^p(\Omega)$ for all $p \in [1,\infty)$, and a.e.\! in $\Omega$. Moreover,
\begin{equation} \label{9}
\sup_j \int_\Omega u_j^+ f(u_j^+)\, dx < \infty
\end{equation}
by \eqref{7} and \eqref{8}, and hence
\begin{equation} \label{10}
\int_\Omega F(u_j^+)\, dx \to \int_\Omega F(u^+)\, dx
\end{equation}
by Lemma \ref{Lemma 4}. Clearly,
\begin{equation} \label{11}
\mu_j \int_\Omega u_j\, g(u_j)\, dx \to \mu \int_\Omega u\, g(u)\, dx, \qquad \mu_j \int_\Omega G(u_j)\, dx \to \mu \int_\Omega G(u)\, dx.
\end{equation}

We claim that the weak limit $u$ is nonzero. Suppose $u = 0$. Then
\begin{equation} \label{12}
\int_\Omega F(u_j^+)\, dx \to 0, \qquad \mu_j \int_\Omega u_j\, g(u_j)\, dx \to 0, \qquad \mu_j \int_\Omega G(u_j)\, dx \to 0
\end{equation}
by \eqref{10} and \eqref{11}, and hence $c > 0$ and
\[
\norm{u_j} \to (Nc)^{1/N}
\]
by \eqref{6}. Let $(Nc)^{1/(N-1)} < \gamma < \alpha_N/\beta$. Then $\norm{u_j} \le \gamma^{(N-1)/N}$ for all $j \ge j_0$ for some $j_0$. Let $q = \alpha_N/\beta \gamma > 1$. By the H\"{o}lder inequality,
\[
\int_\Omega u_j^+ f(u_j^+)\, dx \le \left(\int_\Omega |u_j|^{Np}\, dx\right)^{1/p} \left(\int_\Omega e^{q \beta\, |u_j|^{N'}}\, dx\right)^{1/q},
\]
where $1/p + 1/q = 1$. The first integral on the right-hand side converges to zero since $u = 0$, while the second integral is bounded for $j \ge j_0$ since $q \beta\, |u_j|^{N'} = \alpha_N\, |\widetilde{u}_j|^{N'}$ with $\widetilde{u}_j = u_j/\gamma^{(N-1)/N}$ satisfying $\norm{\widetilde{u}_j} \le 1$, so
\[
\int_\Omega u_j^+ f(u_j^+)\, dx \to 0.
\]
Then $u_j \to 0$ by \eqref{7} and \eqref{12}, and hence $c = 0$ by \eqref{6} and \eqref{12}, a contradiction. So $u$ is nonzero.

Since $E_{\mu_j}'(u_j) \to 0$,
\[
\int_\Omega |\nabla u_j|^{N-2}\, \nabla u_j \cdot \nabla v\, dx - \lambda \int_\Omega f(u_j^+)\, v\, dx + \mu_j \int_\Omega g(u_j)\, v\, dx \to 0
\]
for all $v \in W^{1,N}_0(\Omega)$. For $v \in C^\infty_0(\Omega)$, an argument similar to that in the proof of Lemma \ref{Lemma 4} using the estimate
\[
\abs{\int_{\{u_j^+ \ge M\}} f(u_j^+)\, v\, dx} \le \frac{\sup |v|}{M} \int_\Omega u_j^+ f(u_j^+)\, dx
\]
and \eqref{9} shows that $\dint_\Omega f(u_j^+)\, v\, dx \to \dint_\Omega f(u^+)\, v\, dx$, and $\mu_j \dint_\Omega g(u_j)\, v\, dx \to \mu \dint_\Omega g(u)\, v\, dx$ since $g$ is bounded, so
\[
\int_\Omega |\nabla u|^{N-2}\, \nabla u \cdot \nabla v\, dx = \lambda \int_\Omega f(u^+)\, v\, dx - \mu \int_\Omega g(u)\, v\, dx.
\]
Then this holds for all $v \in W^{1,N}_0(\Omega)$ by density, and taking $v = u$ gives
\begin{equation} \label{13}
\norm{u}^N = \lambda \int_\Omega u^+ f(u^+)\, dx - \mu \int_\Omega u\, g(u)\, dx.
\end{equation}

Next we claim that
\begin{equation} \label{14}
\int_\Omega u_j^+ f(u_j^+)\, dx \to \int_\Omega u^+ f(u^+)\, dx.
\end{equation}
We have
\begin{equation} \label{15}
u_j^+ f(u_j^+) \le |u_j|^N e^{\beta\, |u_j|^{N'}} = |u_j|^N e^{\beta\, \norm{u_j}^{N'} |\widetilde{u}_j|^{N'}},
\end{equation}
where $\widetilde{u}_j = u_j/\norm{u_j}$. Setting $\kappa = \lambda \dint_\Omega F(u^+)\, dx - \mu \dint_\Omega G(u)\, dx$,
\[
\norm{u_j}^N \to N(c + \kappa)
\]
by \eqref{6}, \eqref{10}, and \eqref{11}, so $\widetilde{u}_j$ converges a.e.\! to $\widetilde{u} = u/[N(c + \kappa)]^{1/N}$. Then
\begin{equation} \label{16}
\norm{u_j}^N (1 - \norm{\widetilde{u}}^N) \to N(c + \kappa) - \norm{u}^N.
\end{equation}
By Lemma \ref{Lemma 1} \ref{Lemma 1 (iii)},
\[
\int_\Omega u^+ f(u^+)\, dx \ge N \int_\Omega F(u^+)\, dx,
\]
and it is easily seen that $tg(t) \le N(G(t) + 1/2)$ for all $t \in \R$ and hence
\[
\int_\Omega u\, g(u)\, dx \le N \left(\int_\Omega G(u)\, dx + \half \vol{\Omega}\right),
\]
so it follows from \eqref{13} that $\norm{u}^N \ge N(\kappa - (\mu/2) \vol{\Omega})$. Hence
\begin{equation} \label{17}
N(c + \kappa) - \norm{u}^N \le N \bigg(c + \frac{\mu}{2} \vol{\Omega}\bigg) < \left(\frac{\alpha_N}{\beta}\right)^{N-1}
\end{equation}
by \eqref{4}. We are done if $\norm{\widetilde{u}} = 1$, so suppose $\norm{\widetilde{u}} \ne 1$ and let
\[
\frac{[N(c + (\mu/2) \vol{\Omega})]^{1/(N-1)}}{(1 - \norm{\widetilde{u}}^N)^{1/(N-1)}} < \widetilde{\gamma} - 2 \eps < \widetilde{\gamma} < \frac{\alpha_N/\beta}{(1 - \norm{\widetilde{u}}^N)^{1/(N-1)}}.
\]
Then $\norm{u_j}^{N/(N-1)} \le \widetilde{\gamma} - 2 \eps$ for all $j \ge j_0$ for some $j_0$ by \eqref{16} and \eqref{17}, and
\begin{equation} \label{18}
\sup_j \int_\Omega e^{\beta\, \widetilde{\gamma}\, |\widetilde{u}_j|^{N'}} dx < \infty
\end{equation}
by Lemma \ref{Lemma 5}. For $M > 0$ and $j \ge j_0$, \eqref{15} then gives
\begin{align*}
& \phantom{\le \text{ }} \int_{\{u_j^+ \ge M\}} u_j^+ f(u_j^+)\, dx\\[2pt]
& \le \int_{\{u_j^+ \ge M\}} u_j^N e^{\beta\, (\widetilde{\gamma} - 2 \eps)\, \widetilde{u}_j^{N'}} dx\\[2pt]
& = \norm{u_j}^N \int_{\{u_j^+ \ge M\}} \widetilde{u}_j^N e^{- \eps \beta\, \widetilde{u}_j^{N'}} e^{- \eps \beta\, (u_j/\norm{u_j})^{N'}} e^{\beta\, \widetilde{\gamma}\, \widetilde{u}_j^{N'}} dx\\[2pt]
& \le \left(\max_{t > 0}\, t^N e^{- \eps \beta\, t^{N'}}\right) \norm{u_j}^N e^{- \eps \beta\, (M/\norm{u_j})^{N'}} \int_\Omega e^{\beta\, \widetilde{\gamma}\, \widetilde{u}_j^{N'}} dx.
\end{align*}
The last expression goes to zero as $M \to \infty$ uniformly in $j$ since $\norm{u_j}$ is bounded and \eqref{18} holds, so \eqref{14} now follows as in the proof of Lemma \ref{Lemma 4}.

Now it follows from \eqref{7}, \eqref{14}, \eqref{11}, and \eqref{13} that
\[
\norm{u_j}^N \to \lambda \int_\Omega u^+ f(u^+)\, dx - \mu \int_\Omega u\, g(u)\, dx = \norm{u}^N,
\]
and hence $\norm{u_j} \to \norm{u}$. So $u_j \to u$ by the uniform convexity of $W^{1,N}_0(\Omega)$. Clearly, $E_\mu(u) = c$ and $E_\mu'(u) = 0$.
\end{proof}

\section{Proof of Theorem \ref{Theorem 1}} \label{Section 3}

In this section we prove Theorem \ref{Theorem 1}. Recall that $E_\mu$ satisfies the Palais-Smale compactness condition at the level $c \in \R$, or the \PS{c} condition for short, if every sequence $\seq{u_j}$ in $W^{1,N}_0(\Omega)$ such that $E_\mu(u_j) \to c$ and $E_\mu'(u_j) \to 0$, called a \PS{c} sequence, has a convergent subsequence. The following lemma is immediate from the general compactness result in Theorem \ref{Theorem 2}.

\begin{lemma} \label{Lemma 2}
$E_\mu$ satisfies the {\em \PS{c}} condition for all $c \ne 0$ satisfying
\[
c < \frac{1}{N} \left(\frac{\alpha_N}{\beta}\right)^{N-1} - \frac{\mu}{2} \vol{\Omega}.
\]
\end{lemma}

First we show that $E_\mu$ has a uniformly positive mountain pass level below the threshold for compactness given in Lemma \ref{Lemma 2} for all sufficiently small $\mu > 0$. We may assume that $0 \in \Omega$ without loss of generality. Take $r > 0$ so small that $\closure{B_r(0)} \subset \Omega$ and let
\[
v_j(x) = \frac{1}{\omega_{N-1}^{1/N}}\, \begin{cases}
(\log j)^{(N-1)/N}, & |x| \le r/j\\[10pt]
\dfrac{\log (r/|x|)}{(\log j)^{1/N}}, & r/j < |x| < r\\[10pt]
0, & |x| \ge r.
\end{cases}
\]
It is easily seen that $v_j \in W^{1,N}_0(\Omega)$ with $\norm{v_j} = 1$ and
\begin{equation} \label{20}
\int_\Omega v_j^N\, dx = \O(1/\log j) \quad \text{as } j \to \infty.
\end{equation}

\begin{lemma} \label{Lemma 6}
There exist $\mu_0, \rho, c_0 > 0$, $j_0 \ge 2$, $R > \rho$, and $\vartheta < \dfrac{1}{N} \left(\dfrac{\alpha_N}{\beta}\right)^{N-1}$ such that the following hold for all $\mu \in (0,\mu_0)$:
\begin{enumroman}
\item \label{Lemma 6 (i)} $\norm{u} = \rho \implies E_\mu(u) \ge c_0$,
\item \label{Lemma 6 (ii)} $E_\mu(Rv_{j_0}) \le 0$,
\item \label{Lemma 6 (iii)} denoting by $\Gamma = \set{\gamma \in C([0,1],W^{1,N}_0(\Omega)) : \gamma(0) = 0,\, \gamma(1) = Rv_{j_0}}$ the class of paths joining the origin to $Rv_{j_0}$,
    \begin{equation} \label{21}
    c_0 \le c_\mu := \inf_{\gamma \in \Gamma}\, \max_{u \in \gamma([0,1])}\, E_\mu(u) \le \vartheta + C_\lambda\, \mu^{N'},
    \end{equation}
    where $C_\lambda = (1 - 1/N) \vol{\Omega}/\lambda^{1/(N-1)}$,
\item \label{Lemma 6 (iv)} $E_\mu$ has a critical point $u_\mu$ at the level $c_\mu$.
\end{enumroman}
\end{lemma}

\begin{proof}
Set $\rho = \norm{u}$ and $\widetilde{u} = u/\rho$. By Lemma \ref{Lemma 1} \ref{Lemma 1 (iv)} and since
\[
\lambda_1 = \inf_{u \in W^{1,N}_0(\Omega) \setminus \set{0}}\, \frac{\dint_\Omega |\nabla u|^N\, dx}{\dint_\Omega |u|^N\, dx},
\]
we have
\begin{align*}
\int_\Omega F(u^+)\, dx & \le \int_\Omega \left[\frac{1}{N}\, |u|^N + \frac{\beta}{N}\, |u|^{N^2/(N-1)} e^{\beta\, |u|^{N'}}\right] dx\\[2pt]
& \le \frac{1}{N \lambda_1} \norm{u}^N + \frac{\beta}{N} \pnorm[2N^2/(N-1)]{u}^{N^2/(N-1)} \left(\int_\Omega e^{2 \beta\, |u|^{N'}} dx\right)^{1/2}\\[2pt]
& = \frac{\rho^N}{N \lambda_1} + \frac{\beta \rho^{N^2/(N-1)}}{N} \pnorm[2N^2/(N-1)]{\widetilde{u}}^{N^2/(N-1)} \left(\int_\Omega e^{2 \beta \rho^{N'} |\widetilde{u}|^{N'}} dx\right)^{1/2}\\[2pt]
& = \frac{\rho^N}{N \lambda_1} + \O(\rho^{N^2/(N-1)}) \quad \text{as } \rho \to 0
\end{align*}
since $W^{1,N}_0(\Omega) \hookrightarrow L^{2N^2/(N-1)}(\Omega)$ and $\dint_\Omega e^{2 \beta \rho^{N'} |\widetilde{u}|^{N'}} dx$ is bounded by \eqref{2} when $2 \beta \rho^{N'} \le \alpha_N$. Since $G(t) \ge -1/2$ for all $t \in \R$, then
\[
E_\mu(u) \ge \frac{1}{N} \left(1 - \frac{\lambda}{\lambda_1}\right) \rho^N + \O(\rho^{N^2/(N-1)}) - \frac{\mu}{2} \vol{\Omega}.
\]
Since $\lambda < \lambda_1$, \ref{Lemma 6 (i)} follows from this for sufficiently small $\rho, \mu, c_0 > 0$.

Since $v_j \ge 0$,
\[
E_\mu(tv_j) = \int_\Omega \left[\frac{t^N}{N}\, |\nabla v_j|^N - \lambda F(tv_j) + \mu tv_j\right] dx
\]
for $t \ge 0$. By the H\"{o}lder and Young's inequalities,
\[
\mu t \int_\Omega v_j\, dx \le \mu t \vol{\Omega}^{1-1/N} \left(\int_\Omega v_j^N\, dx\right)^{1/N} \le C_\lambda\, \mu^{N'} + \frac{\lambda t^N}{N} \int_\Omega v_j^N\, dx,
\]
so
\[
E_\mu(tv_j) \le H_j(t) + C_\lambda\, \mu^{N'},
\]
where
\[
H_j(t) = \frac{t^N}{N} \left(1 + \lambda \int_\Omega v_j^N\, dx\right) - \lambda \int_\Omega F(tv_j)\, dx.
\]
By Lemma \ref{Lemma 1} \ref{Lemma 1 (v)},
\[
\int_\Omega F(tv_j)\, dx \ge \frac{t^N}{N} \int_\Omega v_j^N\, dx + \frac{\beta (N - 1)}{N^2}\, t^{N^2/(N-1)} \int_\Omega v_j^{N^2/(N-1)}\, dx,
\]
so
\begin{equation} \label{22}
H_j(t) \le \frac{t^N}{N} - \frac{\lambda \beta (N - 1)}{N^2}\, t^{N^2/(N-1)} \int_\Omega v_j^{N^2/(N-1)}\, dx \to - \infty \quad \text{as } t \to \infty.
\end{equation}
So to prove \ref{Lemma 6 (ii)} and \ref{Lemma 6 (iii)}, it suffices to show that $\exists j_0 \ge 2$ such that
\[
\vartheta := \sup_{t \ge 0}\, H_{j_0}(t) < \frac{1}{N} \left(\frac{\alpha_N}{\beta}\right)^{N-1}.
\]

Suppose $\sup_{t \ge 0}\, H_j(t) \ge (\alpha_N/\beta)^{N-1}/N$ for all $j$. Since $H_j(t) \to - \infty$ as $t \to \infty$ by \eqref{22}, there exists $t_j \ge 0$ such that
\begin{equation} \label{23}
H_j(t_j) = \frac{t_j^N}{N}\, (1 + \eps_j) - \lambda \int_\Omega F(t_j v_j)\, dx = \sup_{t \ge 0}\, H_j(t) \ge \frac{1}{N} \left(\frac{\alpha_N}{\beta}\right)^{N-1}
\end{equation}
and
\begin{equation} \label{24}
H_j'(t_j) = t_j^{N-1} \left(1 + \eps_j - \lambda \int_\Omega v_j^N e^{\beta t_j^{N'}\! v_j^{N'}} dx\right) = 0,
\end{equation}
where
\[
\eps_j = \lambda \int_\Omega v_j^N\, dx.
\]
Since $F(t) \ge 0$ for all $t \ge 0$, \eqref{23} gives
\[
\beta t_j^{N'} \ge \frac{\alpha_N}{1 + \eps_j},
\]
and then \eqref{24} gives
\begin{equation} \label{25}
\frac{1 + \eps_j}{\lambda} = \int_\Omega v_j^N e^{\beta t_j^{N'}\! v_j^{N'}} dx \ge \int_{B_{r/j}(0)} v_j^N e^{\alpha_N v_j^{N'}/(1 + \eps_j)}\, dx = \frac{r^N}{N}\, \frac{(\log j)^{N-1}}{j^{N \eps_j/(1 + \eps_j)}}.
\end{equation}
By \eqref{20}, $\eps_j \to 0$ and
\[
j^{N \eps_j/(1 + \eps_j)} \le j^{N \eps_j} = e^{N \eps_j \log j} = \O(1),
\]
so \eqref{25} is impossible for large $j$.

By \ref{Lemma 6 (i)}--\ref{Lemma 6 (iii)}, $E_\mu$ has the mountain pass geometry and the mountain pass level $c_\mu$ satisfies
\[
0 < c_\mu \le \vartheta + C_\lambda\, \mu^{N'} < \frac{1}{N} \left(\frac{\alpha_N}{\beta}\right)^{N-1} - \frac{\mu}{2} \vol{\Omega}
\]
for all sufficiently small $\mu > 0$, so $E_\mu$ satisfies the \PS{c_\mu} condition by Lemma \ref{Lemma 2}. So $E_\mu$ has a critical point $u_\mu$ at this level by the mountain pass theorem.
\end{proof}

Now we show that $u_\mu$ is positive in $\Omega$, and hence a weak solution of problem \eqref{1}, for all sufficiently small $\mu \in (0,\mu_0)$. It suffices to show that for every sequence $\mu_j > 0,\, \mu_j \to 0$, a subsequence of $u_j = u_{\mu_j}$ is positive in $\Omega$. By \eqref{21}, a renamed subsequence of $c_{\mu_j}$ converges to some $c$ satisfying
\[
0 < c < \frac{1}{N} \left(\frac{\alpha_N}{\beta}\right)^{N-1}.
\]
Then a renamed subsequence of $\seq{u_j}$ converges in $W^{1,N}_0(\Omega)$ to a critical point $u$ of $E_0$ at the level $c$ by Theorem \ref{Theorem 2}. Since $c > 0$, $u$ is nontrivial.

\begin{lemma} \label{Lemma 7}
A further subsequence of $\seq{u_j}$ is bounded in $C^{1,\alpha}_0(\closure{\Omega})$ for some $\alpha \in (0,1)$.
\end{lemma}

\begin{proof}
Since
\[
\left\{\begin{aligned}
- \Delta_N\, u_j & = \lambda f(u_j^+) - \mu_j\, g(u_j) && \text{in } \Omega\\[10pt]
u_j & = 0 && \text{on } \bdry{\Omega},
\end{aligned}\right.
\]
it suffices to show that $\seq{u_j}$ is bounded in $L^\infty(\Omega)$ by the global regularity result of Lieberman \cite{MR969499}, and this will follow from Proposition 1.3 of Guedda and V{\'e}ron \cite{MR1009077} if we show that $f(u_j^+)$ is bounded in $L^s(\Omega)$ for some $s > 1$.

Let $s > 1$. By the H\"{o}lder inequality,
\[
\left(\int_\Omega |f(u_j^+)|^s\, dx\right)^{1/s} \le \left(\int_\Omega |u_j|^p\, dx\right)^{(N-1)/p} \left(\int_\Omega e^{q \beta\, |u_j|^{N'}} dx\right)^{1/q},
\]
where $(N - 1)/p + 1/q = 1/s$. The first integral on the right-hand side is bounded since $W^{1,N}_0(\Omega) \hookrightarrow L^p(\Omega)$, and so is the second integral by Lemma \ref{Lemma 8} below.
\end{proof}

\begin{lemma} \label{Lemma 8}
If $\seq{u_j}$ is a convergent sequence in $W^{1,N}_0(\Omega)$, then
\[
\sup_j \int_\Omega e^{b\, |u_j|^{N'}} dx < \infty
\]
for all $b$.
\end{lemma}

\begin{proof}
The case $b \le 0$ is trivial, so suppose $b > 0$ and let $u \in W^{1,N}_0(\Omega)$ be the limit of $\seq{u_j}$. We have
\[
|u_j|^{N'} \le (|u| + |u_j - u|)^{N'} \le 2^{N'} \big(|u|^{N'} + |u_j - u|^{N'}\big),
\]
so
\[
\int_\Omega e^{b\, |u_j|^{N'}} dx \le \left(\int_\Omega e^{2^{N'+1} b\, |u|^{N'}} dx\right)^{1/2} \left(\int_\Omega e^{2^{N'+1} b\, |u_j - u|^{N'}} dx\right)^{1/2}.
\]
The first integral on the right-hand side is finite, and the second integral equals
\[
\int_\Omega e^{2^{N'+1} b\, \norm{u_j - u}^{N'} |v_j|^{N'}} dx,
\]
where $v_j = (u_j - u)/\norm{u_j - u}$. Since $\norm{v_j} = 1$ and $\norm{u_j - u} \to 0$, this integral is bounded by \eqref{2}.
\end{proof}

By Lemma \ref{Lemma 7}, a renamed subsequence of $u_j$ converges to $u$ in $C^1_0(\closure{\Omega})$. Since $u$ is a nontrivial weak solution of the problem
\[
\left\{\begin{aligned}
- \Delta_N\, u & = \lambda\, (u^+)^{N-1} e^{\beta\, (u^+)^{N'}} && \text{in } \Omega\\[10pt]
u & = 0 && \text{on } \bdry{\Omega},
\end{aligned}\right.
\]
$u > 0$ in $\Omega$ and its interior normal derivative $\partial u/\partial \nu > 0$ on $\bdry{\Omega}$ by the strong maximum principle and the Hopf lemma for the $p$-Laplacian (see V{\'a}zquez \cite{MR768629}). Since $u_j \to u$ in $C^1_0(\closure{\Omega})$, then $u_j > 0$ in $\Omega$ for all sufficiently large $j$. This concludes the proof of Theorem \ref{Theorem 1}.

\subsection*{Acknowledgement}
The second author was supported by the
2018-0340 Research Fund of the University of Ulsan.

\def\cdprime{$''$}

\end{document}